\documentclass[12pt,twoside, a4paper]{amsart}

\usepackage{mathptmx}
\usepackage{amssymb}
\usepackage{eucal}
\usepackage{amsxtra}
\usepackage{verbatim}
\usepackage{enumerate}
\usepackage[english]{babel}

\usepackage[body={17cm,21cm}]{geometry}
\usepackage[T1]{fontenc}
\usepackage{graphicx}
\usepackage{mathrsfs}
\usepackage{amscd,latexsym,amsthm,amsfonts,amssymb,amsmath,amsxtra}
\usepackage[colorlinks, urlcolor=blue,  citecolor=blue]{hyperref}
\usepackage{enumerate}
\usepackage[all]{xy}%
\providecommand{\U}[1]{\protect\rule{.1in}{.1in}}
\RequirePackage{amsmath}
\RequirePackage{amssymb}
\usepackage{comment}

\theoremstyle{plain}

\newtheorem{theorem}{Theorem}
\newtheorem{proposition}[theorem]{Proposition}
\newtheorem{lemma}[theorem]{Lemma}
\newtheorem{corollary}[theorem]{Corollary}

\newtheorem{theorem?}{Theorem(?)} [section]
\newtheorem{proposition?}[theorem]{Proposition(?)}
\newtheorem{lemma?}[theorem]{Lemma(?)}
\newtheorem{corollary?}[theorem]{Corollary(?)}

\newtheorem*{theorem*}{Theorem}
\newtheorem*{proposition*}{Proposition}
\newtheorem*{lemma*}{Lemma}
\newtheorem*{corollary*}{Corollary}
\newtheorem*{question*}{Question}
\newtheorem*{conjecture*}{Conjecture}
\newtheorem*{claim*}{Claim}

\newtheorem*{introtheorem*}{Theorem}
\newtheorem*{introproposition*}{Proposition}
\newtheorem*{introlemma*}{Lemma}
\newtheorem*{introcorollary*}{Corollary}

\theoremstyle{definition}

\newtheorem{definition}[theorem]{Definition}

\newtheorem*{definition*}{Definition}
\newtheorem*{example*}{Example}

\newtheorem{question}[theorem]{Question}

\theoremstyle{remark}
\newtheorem{remark}[theorem]{Remark}

\newtheorem*{remark*}{Remark}

\numberwithin{equation}{section}
\numberwithin{theorem}{section}

\DeclareSymbolFont{rsfs}{U}{rsfs}{m}{n}
\DeclareSymbolFontAlphabet{\mathcal}{rsfs}

\newcommand{\Gal}{{\rm Gal}}
\newcommand{\Pic}{{\rm Pic}}
\newcommand{\Cl}{{\rm Cl}}
\newcommand{\Br}{{\rm Br}}
\renewcommand{\ker}{{\rm ker}}

\newcommand{\Hom}{{\rm Hom}}
\newcommand{\Proj}{{\rm Proj}}
\newcommand{\Spec}{{\rm Spec}}

\newcommand{\lcm}{{\rm lcm}}

\newcommand{\codim}{{\rm codim}}

\newcommand{\Shat}{{\widehat{S}}}

\newcommand{\Gm}{{{\Bbb G}_m}}

\newcommand{\Res}{{\rm Res}}

\def\Div{{\textup{Div}}}

\begin{document}	
\title[]
{Arithmetic purity of strong approximation for complete toric varieties}

\author{Sheng Chen}

\address{Sheng Chen \newline School of Mathematical Sciences, \newline University of Science and Technology of China,
\newline 96 Jinzhai Road, \newline 230026 Hefei, China}
\email{chenshen1991@163.com}

\date{\today}


\begin{abstract}
In this article, we establish the arithmetic purity of strong approximation for smooth loci of weighted projective spaces. By using this result and the descent method, we also prove that the arithmetic purity of strong approximation with Brauer--Manin obstruction holds for
any smooth and complete toric variety.
\end{abstract}

\maketitle
\section{Introduction}
\subsection{Background}
Let $k$ be a number field. It's known that weak approximation is a $k$-birational invariant of smooth geometrically $k$-varieties. However, this is not true for strong approximation. In fact, Min\v{c}hev in \cite{Minchev} pointed out that strong approximation cannot be true for varieties which are not simply connected. However, one can expect that strong approximation is invariant among smooth varieties up to a closed subvariety of codimension $\ge 2$. Wei \cite[Lemma 2.1]{Wei14}, and independently Cao and Xu \cite[Proposition 3.6]{CX}
 proved that this is true for affine spaces. Based on this evidence, Wittenberg \cite[Question 2.11]{Wit16} proposed the following question.
\begin{question}\label{q1}
Let $X$ be a smooth variety over a number field which satisfies strong approximation off a finite set of places of $k$. Does any Zariski open subset $U$ of $X$ also satisfy this property, whenever $\codim(X\setminus U,X)\geqslant 2$?
\end{question}
In \cite{CLX}, Cao, Liang and Xu gave an affirmative answer to Question \ref{q1} for semi-simple simply connected groups that are quasi-split. Cao and Huang \cite{CH} provided further positive answers to Question \ref{q1} for semi-simple simply connected groups.

On the other hand, a similar question for strong approximation with Brauer--Manin obstruction was proposed by Cao, Liang and Xu \cite[Question 1.2]{CLX}.	
\begin{question}\label{q2}
Let $X$ be a smooth geometrically integral variety over a number field $k$ and $S$ be a finite set of places of $k$.
Suppose that $\Pic(X_{\bar{k}})$ is finitely generated and $\bar{k}[X]^\times =\bar{k}^{\times}$ where $\bar{k}$ is an algebraic closure of $k$. If $X$ satisfies strong approximation with Brauer--Manin obstruction off $S$,  does any open subvariety of $X$ with complement of codimension $\geq 2$ satisfy the same property?
\end{question}

Wei \cite{Wei14} gave an affirmative answer to Question \ref{q2} for smooth toric varieties when $S\ne \emptyset$. In \cite{CLX}, Cao, Liang and Xu extended this result to partial smooth equivariant compactifications of homogeneous spaces.

However, in the above work on Question \ref{q1} and \ref{q2}, the finite set $S$ of places of $k$ cannot be empty. To the knowledge of the author, the affirmative answers to Question \ref{q1} and \ref{q2} may be known only for projective spaces when $S=\emptyset$. In this paper, we will give an affirmative answer to Question \ref{q2} for complete toric varieties (see Theorem \ref{t.2.3}) when $S=\emptyset$ and then we provide some other positive evidence (see Theorem \ref{t.2.5}) for Question \ref{q1}. Following Wei \cite{Wei14}, we use the descent theory and the
combinatorial description of toric varieties to achieve our goals.

\subsection{Notation and terminology} Let $k$ be a number field, $\Omega_k$ be the set of all places of $k$ and $\infty_k$ be
the set of all infinite places of $k$. Let $\mathcal{O}_k$ be the ring of integers of $k$ and $\mathcal{O}_S$ be the ring of $S$-integers of $k$ for a finite set $S$ of $\Omega_k$ containing $\infty_k$. For each $v \in \Omega_k$, the completion of $k$ at $v$ is denoted by $k_v$ and the completion of $\mathcal{O}_k$ at $v$ is denoted by $\mathcal{O}_v$. The ring of adeles of $k$ is denoted by ${\mathbf A}_k$. The ring of adeles without $S$-components is denoted by $\mathbf A^{S}_k$.

Let $\bar{k}$ be a fixed algebraic closure of $k$ and we write $\Gamma_k=\Gal(\bar k/k)$. For any scheme $X$ over $k$, we write $X_{\bar k}=X\times_k \bar{k}$, and we denote by $X_{sm}$ its smooth locus. Let $$\Br(X)=H_{\text{\'et}}^2(X, \Bbb G_m),  \ \ \ \Br_1(X)= \ker[\Br(X)\rightarrow \Br(X_{\bar{k}})], \ \ \ \Br_a(X)=\Br_1(X)/\Br(k). $$

Let $f:X \to Y$ be a morphism of schemes. For an irreducible closed subset $A\subset X$ and an irreducible closed subset $B\subset Y$, if $\overline{f(A)}= B$, we say that $A$ dominates $B$.

Let $R=\bigoplus_{n\in \Bbb Z}R_n$ be a graded ring. For a positive integer $d$, we write $R^{(d)}=\bigoplus_{n\in \Bbb Z}R_{dn}$.

By convention, we also regard $\emptyset$ as a finite set of $\Omega_k$.
\begin{definition}
Let $X$ be a geometrically integral $k$-variety and $S$ be a finite set of $\Omega_k$ and $pr^S$ be the projection $X(\mathbf A_k) \to X(\mathbf A^{S}_k)$.

(1) If $X(k)$ is dense in $pr^S(X({\mathbf A}_k))$, we say that $X$ satisfies strong approximation off $S$.

(2) If $X(k)$ is dense in $ pr^S(X({\mathbf A}_k)^{B})$ for some subset $B$ of $\Br(X)$, we say that $X$ satisfies strong approximation with respect to $B$ off $S$.
\end{definition}
When $S=\emptyset$, we simply say that $X$ satisfies strong approximation or satisfies strong approximation with respect to $B$.

\begin{definition}
Let $X$ be a geometrically integral $k$-variety and $S$ be a finite set of $\Omega_k$.

(1) We say that $X$ satisfies the arithmetic purity off $S$ if any open subvariety with complement of codimension $\ge 2$ satisfies
strong approximation off $S$.

(2) We say that $X$ satisfies the arithmetic purity with respect to $B \subset \Br(X)$ off $S$ if any open subvariety with complement of codimension $\ge 2$ satisfies strong approximation with respect to $B$ off $S$.
\end{definition}
When $S=\emptyset$, we simply say that $X$ satisfies the arithmetic purity or satisfies the arithmetic purity with respect to $B$.
\subsection{Main results}

The main results of this paper are the following two theorems.

\begin{theorem} \label{t.2.3}
Let $k$ be a number field, and $T$ be a torus over $k$. Let $X$ be one of the following varieties:

(1) complete \footnote{Here, we say that $X$ is simplicial if $X_{\bar k}$ is simplicial. By \cite[Proposition 4.2.7]{CJH}, $X_{\bar k}$ is simplicial if and only if $X_{\bar k}$ is normal and $\Pic(X_{\bar k})$ has finite index in $\Cl(X_{\bar k})$.}simplicial $T$-toric varieties over $k$;

(2) complete normal $T$-toric varieties over $k$ if $T$ is a $k$-split torus;

(3) complete normal $T$-toric varieties over $k$ if $T$ is a $k$-anisotropic torus.

Let $Z \subset X_{sm}$ be a closed subset of codimension $\ge 2$. Then $X_{sm} \setminus Z$ satisfies strong approximation with respect to $\Br_1(X_{sm})$.
\end{theorem}

\begin{theorem}\label{t.2.5}
Let $X$ be a complete, smooth $T$-toric variety over $k$ and $S\subset \Omega_k$ be a finite set of places of $k$ such that $T(k)$ is dense in $\prod_{v\in {\Omega_k\setminus S}} T(k_v)$. Then $X$ satisfies the arithmetic purity off $S$.
\end{theorem}
\begin{remark}

(1) By \cite[Theorem 7.7]{PR}, such a finite set $S \subset \Omega_k$ (in Theorem \ref{t.2.5}) does exist. If $T$ satisfies weak approximation, then we can take $S=\emptyset$.

(2) In Section 3, we will show that Theorem \ref{t.2.5} is an easy consequence of Theorem \ref{t.2.3}.
\end{remark}

The organization of the paper is as follows. In Section 2, we establish the arithmetic purity for smooth loci of weighted projective spaces. In Section 3, we give the proofs of the two theorems by using the results of Section 2 and the descent theory.

\section{Arithmetic purity for smooth loci of weighted projective spaces}
Let $q_0, q_1,\cdots, q_n$ be positive integers and set $$d=\gcd(q_0, q_1,\cdots, q_n), \ \  m=\lcm(q_0, q_1,\cdots, q_n).$$
Let
$$\mathbb{P}_k(q_0, q_1,...,q_n):=\Proj(k[X_0,\cdots, X_n]),$$
where the gradation is defined by $\deg X_i=q_i$ for $0 \leqslant i\leqslant n$. Without loss of generality, we always assume $d=1$ (cf. \cite[Lemma 3A.3]{wps}). The projective scheme $\mathbb{P}_k(q_0, q_1,...,q_n)$ is called a weighted projective space. We shall denote by $\mathbb{P}_k^n$ the usual projective space i.e. $\mathbb{P}_k(1, 1,...,1)$.

It is well known that a $\Bbb Z$-graduation of a commutative ring is equivalent to an action of the torus ${\Bbb G}_m$ on its spectrum.
In our case, ${\Bbb G}_m$ acts on ${\Bbb A}_k^{n+1}=\Spec(k[X_0,\cdots, X_n])$ as follows
$${\Bbb G}_m \times {\Bbb A}_k^{n+1} \to {\Bbb A}_k^{n+1}; \ \ (t,(x_0,\cdots, x_n)) \mapsto (t^{q_0}x_0,\cdots, t^{q_n}x_n). $$
Set $U={\Bbb A}_k^{n+1}-\{0\}$. The geometric quotient $U/{\Bbb G}_m$ exists and coincides with $\mathbb{P}_k(q_0, q_1,...,q_n)$ (cf. \cite[Theorem 3A.1]{wps}).

Note that the weighted projective space $\mathbb{P}_k(q_0, q_1,...,q_n)$ is also a toric variety
(cf. \cite[Example 5.1.14]{CJH}). In particular, it's normal by \cite[Theorem 3.1.5]{CJH}. For introductions to toric varieties over an arbitrary field, we refer to \cite{Bat} and \cite{Hur}.

\begin{definition}(\cite[Definition 1.2]{Mori})

For any $h \in {\Bbb Z}_{>1}$, we denote by $S_h$ the closed subset $V_+(\{X_i|h\nmid q_i, 0\leqslant i\leqslant n\})\subset \mathbb{P}_k(q_0, q_1,...,q_n)$.
Set $$\mathbb{P}_k^*(q_0, q_1,...,q_n):=\mathbb{P}_k(q_0, q_1,...,q_n)-\bigcup_{h\in {\Bbb Z}_{>1}}S_h.$$
We call the scheme $\mathbb{P}_k^*(q_0, q_1,...,q_n)$ a weak projective space. We write
$\mathrm{O}^*(l)=\mathrm{O}(l)|_{\mathbb{P}_k^*(q_0, q_1,...,q_n)}$ for every integer $l$, where $\mathrm{O}(l)$ is the twisting sheaf on $\mathbb{P}_k(q_0, q_1,...,q_n).$
\end{definition}

\begin{remark}\label{rem.2.0}
By \cite[Proposition 5.6, Proposition 5.7]{wps}, the weak projective space $\mathbb{P}_k^*(q_0, q_1,...,q_n)$ has the following properties,
\begin{enumerate}
  \item for every $l\in \Bbb Z$, the sheaf $\mathrm{O}^*(l)$ is invertible,
  \item $\mathrm{O}^*(1)^{\otimes l}\cong \mathrm{O}^*(l)$ for every $l\in \Bbb Z$,
  \item $\mathbb{P}_k^*(q_0, q_1,...,q_n)$ is the largest open subset of $\mathbb{P}_k(q_0, q_1,...,q_n)$ with the properties (1) and (2),
  \item $\mathbb{P}_k^*(q_0, q_1,...,q_n)$ is smooth.

\end{enumerate}
\end{remark}

\begin{lemma}\label{lem.2.0.1}
Set $\mathbb{A}_k^*(q_0, q_1,...,q_n):=\Spec(k[X_0,\cdots, X_n])-\bigcup_{h\in {\Bbb Z}_{>1}}V(\{X_i|h\nmid q_i, 0\leqslant i\leqslant n\})$. Let $$\pi: \mathbb{A}_k^*(q_0, q_1,...,q_n) \to \mathbb{P}_k^*(q_0, q_1,...,q_n)$$
be the natural morphism.

(1) Let $Z$ be a closed subset of $\mathbb{P}_k^*(q_0, q_1,...,q_n)$. Then we have
$$\pi((\mathbb{A}_k^*(q_0, q_1,...,q_n)\setminus{\pi^{-1}(Z)})({\mathbf A}_k))=(\mathbb{P}_k^*(q_0, q_1,...,q_n)\setminus Z)({\mathbf A}_k).$$

(2) Set $r= \mathop{\min}\limits_{h\in {\Bbb Z}_{>1}} \#\{X_i|h\nmid q_i, 0\leqslant i\leqslant n\}$. If $r>1$, then $$\Pic(\mathbb{P}_{\bar k}^*(q_0, q_1,...,q_n))\cong \Bbb Z.$$

\end{lemma}
\begin{proof}
One can easily check that any geometric point of $\mathbb{A}_k^*(q_0, q_1,...,q_n)$ has trivial stabilizer in ${\Bbb G}_m$. By \cite[Proposition 2.3]{Mori}, the morphism $\pi$ is faithfully flat. It is immediate to check that $\mathbb{A}_k^*(q_0, q_1,...,q_n)$ is a $\mathbb{P}_k^*(q_0, q_1,...,q_n)$-torsor under
${\Bbb G}_m$ and then $\mathbb{A}_k^*(q_0, q_1,...,q_n)\setminus{\pi^{-1}(Z)}$ is also a torsor of $\mathbb{P}^*(q_0, q_1,...,q_n)\setminus Z$ under ${\Bbb G}_m$. By Hilbert's Theorem 90 and \cite[Chapter II, Proposition 6.2]{Hartshorne}, we have
$$H_{\text{\'et}}^1(k_v, \Bbb G_m)=0, \ \  H_{\text{\'et}}^1(\mathcal{O}_v, \Bbb G_m)=0$$
for $v \in \Omega_k$. Thus we obtain that $$\pi((\mathbb{A}_k^*(q_0, q_1,...,q_n)\setminus{\pi^{-1}(Z)})({\mathbf A}_k))=(\mathbb{P}_k^*(q_0, q_1,...,q_n)\setminus Z)({\mathbf A}_k).$$

By \cite[Proposition 2.1.1]{CTS87},  we have the following exact sequence of $\Gamma_k$-modules
$$0 \to {\bar k}[\mathbb{P}_k^*(q_0, q_1,...,q_n)]^\times \to {\bar k}[\mathbb{A}_k^*(q_0, q_1,...,q_n)]^\times \to {\Bbb Z} \to \Pic(\mathbb{P}_{\bar k}^*(q_0, q_1,...,q_n))\to 0.$$
If $r>1$, we have
$$\codim(\bigcup_{h\in {\Bbb Z}_{>1}}V(\{X_i|h\nmid q_i, 0\leqslant i\leqslant n\}),\ {\Bbb A}_k^{n+1})=\codim(\bigcup_{h\in {\Bbb Z}_{>1}}S_h,\ \mathbb{P}_k(q_0, q_1,...,q_n))>1,$$
hence $${\bar k}[\mathbb{P}_k^*(q_0, q_1,...,q_n)]^\times={\bar k}[\mathbb{A}_k^*(q_0, q_1,...,q_n)]^\times={\bar k^{\times}}.$$
Thus we obtain that $\Pic(\mathbb{P}_{\bar k}^*(q_0, q_1,...,q_n))\cong \Bbb Z$ (Mori has obtained this result through a different method, cf. \cite[Proposition 2.3]{Mori}).
\end{proof}
\begin{remark}\label{rem 2.0.1.1}
A weighted projective space with $r>1$ is called a normalized weighted projective space. Any weighted projective space is isomorphic to a
normalized weighted projective space (cf. \cite[Section 3C]{wps}). In particular, by \cite[Proposition 3C.5]{wps}, we have $\mathbb{P}_k(1, q_1,...,q_n)\cong \mathbb{P}_k(1, \frac{q_1}{a},...,\frac{q_n}{a})$, where $a=\gcd(q_1,...,q_n)$.
\end{remark}

The following two lemmas are well known.
\begin{lemma}\label{lem.2.0.2}
Let $\alpha : X_1 \to X_2$ be a morphism of smooth geometrically integral varieties over $k$ satisfying $\overline{\alpha(X_1({\mathbf A}_k))}=X_2({\mathbf A}_k)$ (this is the case if $\alpha$ has a rational section over $k$). If $X_1$ satisfies strong approximation, then $X_2$ satisfies strong approximation.
\end{lemma}
\begin{lemma}\label{lem.2.0.3}
Let X be a smooth geometrically integral $k$-variety and $U$ be an open subset of X with complement of codimension $\ge 2$.
Let $S$ be a finite set of $\Omega_k $. Then
\begin{center}
 $U$ satisfies the arithmetic purity off $S \iff X$ satisfies the arithmetic purity off $S$.
\end{center}
\end{lemma}
\begin{proof}
This follows from \cite[Proposition 2.3]{CTX13}.
\end{proof}


\begin{proposition} \label{prop.2.1}
For any weighted projective space $\mathbb{P}_k(q_0, q_1,...,q_n)=\Proj(k[X_0,\cdots, X_n])$ with $q_0=1$, its smooth locus satisfies the arithmetic purity.
\end{proposition}
\begin{proof}
Without loss of generality, let $\mathbb{P}_k(q_0, q_1,...,q_n)$ be a normalized weighted projective space (see Remark \ref{rem 2.0.1.1}).

Set $\Xi:=\#\{q_i|q_i>1,0 \leqslant i\leqslant n \}$. We prove this proposition by induction on $\Xi$.
\begin{enumerate}
\item Case $\Xi=0$.

The proof is similar to \cite[Lemma 2.1]{Wei14}. By \cite[Proposition 2.3]{CTX13}, we only need to show that $U=\mathbb{P}_k^n \setminus Z$ satisfies strong approximation, where $Z \subset\mathbb{P}_k^n$ is a closed subset with all irreducible components of codimension 2.
Fix a finite set $S \subset \Omega_k$ containing $\infty_k$, an integral model $\mathcal{U}$ of $U$ over $\mathcal{O}_S$ and a family $(P_v)_{v \in \Omega_k} \in \prod_{v\in S}U(k_v)\times \prod_{v\notin S}\mathcal{U}(\mathcal{O}_v)$. It suffices to find a closed subscheme $W$ of $U$ satisfying strong approximation, such that
\begin{enumerate}
\item $W(k_v)$ contains a point $P'_v$ very close to $P_v$ for all $v \in S$, \\
\item $W(k_v) \bigcap \mathcal{U}(\mathcal{O}_v) \ne \emptyset$ for all $v \notin S$.\\
\end{enumerate}

We choose a point $T \in U(k)$ arbitrarily close to $P_v$ for $v \in S$ by weak approximation. Then there exists a finite set $S' \subset \Omega_k \setminus S$ such that for $v \notin S'\cup S$, we have $T \in \mathcal{U}(\mathcal{O}_v)$. In the case $n \ge 3$, we choose
a point $T' \in U(k)$ very close to $P_v$ for $v \in S'$. Then we claim that there exists a hyperplane $H$ of $\mathbb{P}_k^n$ not containing any irreducible component of $Z$, such that $T, T' \in H$. The claim can be proved as follows. Let
 $L_0=\{f\in k[X_0,\cdots, X_n]|$ $f$ is a linear homogeneous polynomial\}. Let
 $$L_1=\{f\in L_0| f(T)=0, f(T')=0\}\cup \{0\}.$$
Then $\dim_k (L_1)\ge 2$. We write $Z=\bigcup_{\alpha\in I} Z_\alpha$, where each $Z_\alpha$ is an irreducible component of $Z$. For every $\alpha \in I$, let
\begin{center}
$L_\alpha=\{f \in L_0|f(Z_\alpha)=0\}\cup \{0\}$.
\end{center}
Then $\dim_k (L_\alpha)\le 2$ for any $\alpha \in I$. Since $T\notin Z$, one concludes that $\dim_k(L_\alpha\cap L_1)\le 1$ for any $\alpha \in I$. Thus such
a hyperplane $H$ exists.

Note that $H \cap U$ satisfies (a) and (b), it suffices to show that $H \cap U$ satisfies strong approximation. Note that the codim$(H\setminus H \cap U,H)=2$, thus we reduce the proof to the case $n=2$. Now
we choose a point $T' \in U(k)$ very close to $P_v$ for $v \in S'$, not contained in the union of projective lines through $T$ meeting $Z$. Then we can take $W$ to be the projective line through $T$ and $T'$. Now the case of $\Xi=0$ is established.
\item For $\Xi>0$, we build an induction. Without loss of generality, we assume

(1) $q_i=1$ for $0\leqslant i\leqslant \mu-1$, where $\mu \ge1$,

(2)$q_i>1$ for $\mu\leqslant i\leqslant n$.\\
Consider the scheme $\mathbb{P}_k(q'_0, q'_1,\cdots, q'_{n+1})=\Proj(k[Y_0,\cdots, Y_{n+1}])$, where

\begin{center}
$q'_j=1$ for $0 \leqslant j\leqslant {\mu +1}$, \ \ $q'_j=q_{j-1}$ for $ {\mu +2} \leqslant j\leqslant {n+1}.$
\end{center}
Let\\

$D_+(X_0)=\Spec (k[\frac{X_0}{X_0},\cdots,\frac{X_{\mu -1}}{X_0},\frac{X_{\mu }}{X^{q_{\mu}}_0},\frac{X_{\mu +1}}{X^{q_{\mu +1}}_0},\cdots,\frac{X_{n}}{X^{q_n}_0} ]) \subset \mathbb{P}_k^*(q_0, q_1,...,q_n)$,

$D_+(Y_0)=\Spec (k[\frac{Y_1}{Y_0},\cdots,\frac{Y_{\mu +1}}{Y_0},\frac{Y_{\mu +2}}{Y^{q_{\mu +1}}_0},\frac{Y_{\mu +3}}{Y^{q_{\mu +2}}_0},\cdots,\frac{Y_{n+1}}{Y^{q_n}_0} ]) \subset \mathbb{P}_k^*(q'_0, q'_1,\cdots, q'_{n+1})$,

$D_+(Y_1)=\Spec (k[\frac{Y_0}{Y_1},\frac{Y_1}{Y_1},\cdots,\frac{Y_{\mu}}{Y_1}, \frac{Y_{\mu +1}}{Y_1},\frac{Y_{\mu +2}}{Y^{q_{\mu +1}}_1},\frac{Y_{\mu +3}}{Y^{q_{\mu +2}}_1},\cdots,\frac{Y_{n+1}}{Y^{q_n}_1} ]) \subset \mathbb{P}_k^*(q'_0, q'_1,\cdots, q'_{n+1})$.

Define the map
\begin{center}
$\phi: k[X_0,\cdots, X_n] \to k[\frac{Y_1}{Y_0},\cdots,\frac{Y_{\mu +1}}{Y_0},\frac{Y_{\mu +2}}{Y^{q_{\mu +1}}_0},\frac{Y_{\mu +3}}{Y^{q_{\mu +2}}_0},\cdots,\frac{Y_{n+1}}{Y^{q_n}_0} ]$

$X_i \mapsto \frac{Y_{i+1}}{Y^{q'_{i+1}}_0}$ \  ($0 \leqslant i \leqslant n$).
\end{center}
Then $\phi$ induces an open immersion
$$\phi^*:\mathbb{A}_k^*(q_0, q_1,...,q_n)\hookrightarrow D_+(Y_0).$$
We identity $\mathbb{A}_k^*(q_0, q_1,...,q_n)$ as an open subset of $\mathbb{P}_k^*(q'_0, q'_1,\cdots, q'_{n+1})$. We claim:

There exists an open subset $V\subset \mathbb{P}_k^*(q'_0, q'_1,\cdots, q'_{n+1})$ with complement of codimension $\ge 2$, such that
the natural morphism $\pi: \mathbb{A}_k^*(q_0, q_1,...,q_n) \to \mathbb{P}_k^*(q_0, q_1,...,q_n)$ can be extended to a morphism $\Pi :V \to \mathbb{P}_k^*(q_0, q_1,...,q_n)$ satisfying that for any closed subset $E \subset \mathbb{P}_k^*(q_0, q_1,...,q_n)$ of codimension $\ge 2$, the closed subset $\Pi^{-1}(E)$ has codimension $\ge 2$ in $V$.

If such $V$ and morphism $\Pi$ exist, we can see that
\begin{center}
$\mathbb{P}_k(q'_0, q'_1,\cdots, q'_{n+1})_{sm}$ satisfies the arithmetic purity  $\Longrightarrow$ $\mathbb{P}_k(q_0, q_1,...,q_n)_{sm}$ satisfies the arithmetic purity

\end{center}
by Lemma \ref{lem.2.0.1}, \ref{lem.2.0.2} and \ref{lem.2.0.3}.
Note that $\Xi(\mathbb{P}_k(q'_0, q'_1,\cdots, q'_{n+1}))=\Xi(\mathbb{P}_k(q_0, q_1,...,q_n))-1$, the induction can be established.

Now we prove this claim. Denote by $\hat{\phi}$ the map of function fields induced by $\pi $
\begin{center}
$\hat{\phi}:K(D_+(X_0)) \to K((D_+(Y_1)).$
\end{center}
One can check that
\begin{center}
 $\hat{\phi}(k[\frac{X_0}{X_0},\cdots,\frac{X_{\mu -1}}{X_0},\frac{X_{\mu }}{X^{q_{\mu}}_0},\frac{X_{\mu +1}}{X^{q_{\mu +1}}_0},\cdots,\frac{X_{n}}{X^{q_n}_0} ])= k[\frac{Y_1}{Y_1},\cdots,\frac{Y_{\mu}}{Y_1}, \frac{Y_{\mu +1}}{Y_1}(\frac{Y_0}{Y_1})^{q_{\mu}-1},\frac{Y_{\mu +2}}{Y^{q_{\mu +1}}_1},\frac{Y_{\mu +3}}{Y^{q_{\mu +2}}_1},\cdots,\frac{Y_{n+1}}{Y^{q_n}_1} ]$.
\end{center}
Hence
\begin{center}
$\hat{\phi}(k[\frac{X_0}{X_0},\cdots,\frac{X_{\mu -1}}{X_0},\frac{X_{\mu }}{X^{q_{\mu}}_0},\frac{X_{\mu +1}}{X^{q_{\mu +1}}_0},\cdots,\frac{X_{n}}{X^{q_n}_0} ])\subset k[\frac{Y_0}{Y_1}, \frac{Y_1}{Y_1},\cdots,\frac{Y_{\mu}}{Y_1}, \frac{Y_{\mu +1}}{Y_1},\frac{Y_{\mu +2}}{Y^{q_{\mu +1}}_1},\frac{Y_{\mu +3}}{Y^{q_{\mu +2}}_1},\cdots,\frac{Y_{n+1}}{Y^{q_n}_1} ].$
\end{center}
We thus get a morphism $\pi': D_+(Y_1) \to D_+(X_0)$, in particular the divisor defined by $Y_0=0$ dominates the divisor defined by $X_\mu=0$. Since $\pi$ is flat, we obtain that $\pi^{-1}(E)$ has codimension $\ge 2$ in $\mathbb{A}_k^*(q_0, q_1,...,q_n)$
for any closed subset $E \subset \mathbb{P}_k^*(q_0, q_1,...,q_n)$ of codimension $\ge 2$. Since the complement of
$\mathbb{A}_k^*(q_0, q_1,...,q_n)$ in $D_+(Y_0)$ has codimension $\ge 2$, then we can take $V=\mathbb{A}_k^*(q_0, q_1,...,q_n) \bigcup D_+(Y_1)$. The morphisms $\pi $ and $\pi'$ are obviously patched together and we obtain $\Pi$.

\end{enumerate}
\end{proof}

\begin{remark}
Proposition \ref{prop.2.1} is also true for any weighted projective space, see Corollary \ref{cor.2.4.1}.
\end{remark}

\begin{corollary}\label{G/P}
Let $G$ be a reductive group over $k$, and $P$ be a parabolic subgroup of $G$ over $k$. Then $G/P$ satisfies the arithmetic purity.
\end{corollary}
\begin{proof}
By the proof of \cite[Proposition 20.5]{Bor}, we obtain that $G/P$ contains an open subset $U$ which is isomorphic to ${\Bbb A}_k^{n}$, where $n=\dim(G/P)$. Then we have a rational map
$$j :\mathbb{P}_k^n \dashrightarrow G/P$$
induced by $U\cong {\Bbb A}_k^{n}\subset \mathbb{P}_k^n$. Since $G/P$ is projective and $\mathbb{P}_k^n$ is normal, there is a closed subset $Z \subset \mathbb{P}_k^n$ of codimension $\ge2$ such that $j:{\mathbb{P}_k^n \setminus Z} \to G/P$ is a morphism. Set
$$C:=\mathbb{P}_k^n \setminus U.$$
We denote by $D$ the closure of $j(C\cap(\mathbb{P}_k^n \setminus Z))$ in $G/P$. Note that $D$ has $k$-points and we take a $k$-point $\bar{q}_1 \in D(k)$. By \cite[Proposition 20.5]{Bor}, the map
 \begin{equation}\label{G}
   \phi:G(k)\to (G/P)(k)
 \end{equation}
is surjective, hence there is a $k$-point ${q}_1 \in G(k)$ such that $\phi(q_1)=\bar{q}_1$.

Now we show that $(G/P)\setminus W$ satisfies strong approximation for any closed subset $W\subset G/P$ of codimension $\ge2$.
Since $(G/P)(k)$ is dense in $G/P$, we can take a $k$-point $\bar{q}_2 \in (G/P)(k)$ not belonging to $W$. By surjectivity of (\ref{G}), there is a $k$-point ${q}_2 \in G(k)$ such that $\phi(q_2)=\bar{q}_2$ . Then we have $q_2q_1^{-1}D \not \subset W$, and hence $D \not \subset q_1q_2^{-1}W$. Since $$(G/P)\setminus W \cong (G/P)\setminus q_1q_2^{-1}W,$$ we only need to show that
$(G/P)\setminus q_1q_2^{-1}W$ satisfies strong approximation. Consider the following map
$$j: \mathbb{P}_k^n \setminus (Z \cup j^{-1}(q_1q_2^{-1}W)) \to (G/P)\setminus q_1q_2^{-1}W,$$
we have
 $$\codim(j^{-1}(q_1q_2^{-1}W), \mathbb{P}_k^n \setminus Z) \ge 2, \ \ \overline{j(\mathbb{P}_k^n \setminus (Z \cup j^{-1}(q_1q_2^{-1}W)))({\mathbf A}_k)}=((G/P)\setminus q_1q_2^{-1}W)({\mathbf A}_k).$$
Hence $(G/P)\setminus q_1q_2^{-1}W$ satisfies strong approximation by Lemma \ref{lem.2.0.2}, \ref{lem.2.0.3} and Proposition \ref{prop.2.1}.

\end{proof}

\begin{proposition}\label{prop.2.2.0}
Let $X$ be a $k$-form of $\mathbb{P}_k(q_0, q_1,...,q_n)$. If $X_{sm}(k)\ne \emptyset$, then $X\cong \mathbb{P}_{k}(q_0, q_1,...,q_n)$.
\end{proposition}
\begin{proof}
Without loss of generality, let $\mathbb{P}_k(q_0, q_1,...,q_n)$ be a normalized weighted projective space (see Remark \ref{rem 2.0.1.1}).

Consider the natural morphism
$$\bar{\pi}: \mathbb{A}_{\bar{k}}^*(q_0, q_1,...,q_n) \to \mathbb{P}_{\bar{k}}^*(q_0, q_1,...,q_n).$$
By \cite[Lemma 3B.2 (b)]{wps}, we have
$$\bar{\pi}(\mathscr{O}_{\mathbb{A}_{\bar{k}}^*(q_0, q_1,...,q_n)})=\bigoplus_{l\in \Bbb Z}{\mathrm{O}^*(l)}.$$
Since $\codim({\Bbb A}_{\bar{k}}^{n+1} \setminus \mathbb{A}_{\bar{k}}^*(q_0, q_1,...,q_n), {\Bbb A}_{\bar{k}}^{n+1}) \ge 2$, we obtain that
$$H^0(\mathbb{A}_{\bar{k}}^*(q_0, q_1,...,q_n),\mathscr{O}_{\mathbb{A}_{\bar{k}}^*(q_0, q_1,...,q_n)})=\bar{k}[X_0,\cdots, X_n].$$
Thus we have
\begin{equation}\label{A.0}
\bigoplus_{l\in \Bbb Z}H^0(\mathbb{P}_{\bar{k}}^*(q_0, q_1,...,q_n),\mathrm{O}^*(l))\cong \bar{k}[X_0,\cdots, X_n].
\end{equation}

Since codim$(\mathbb{P}_{\bar{k}}(q_0, q_1,...,q_n)_{sm}\setminus \mathbb{P}_{\bar{k}}^*(q_0, q_1,...,q_n), \mathbb{P}_{\bar{k}}(q_0, q_1,...,q_n)_{sm}) \ge 2$, we have

$$\Pic(\mathbb{P}_{\bar{k}}(q_0, q_1,...,q_n)_{sm}) \cong \Pic(\mathbb{P}_{\bar{k}}^*(q_0, q_1,...,q_n))\cong \Bbb Z$$
by Lemma \ref{lem.2.0.1}. Note that $\mathrm{O}^*(1)$ generates $\Pic(\mathbb{P}_{\bar{k}}^*(q_0, q_1,...,q_n))$ by \cite[Proposition 2.3]{Mori} and $\mathrm{O}^*(1)^{\otimes l}\cong \mathrm{O}^*(l)$ for every $l\in \Bbb Z$ (see Remark \ref{rem.2.0}). Then there exists a line bundle $\overline{\mathcal{F}}$ on $\mathbb{P}_{\bar{k}}(q_0, q_1,...,q_n)_{sm}\cong (X_{\bar k})_{sm}$ satisfying $\overline{\mathcal{F}}|_{\mathbb{P}_{\bar{k}}^*(q_0, q_1,...,q_n)}\cong \mathrm{O}^*(1)$ and we have
\begin{equation}\label{A.1}
\bigoplus_{l\in \Bbb Z}H^0(\mathbb{P}_{\bar{k}}(q_0, q_1,...,q_n)_{sm},\overline{\mathcal{F}}^{\otimes{l}})\cong \bar{k}[X_0,\cdots, X_n]
\end{equation}
by (\ref{A.0}). Since $\overline{\mathcal{F}}^{\otimes{(-1)}}$ has no global sections, we conclude that the action of $\Gamma_k$ on $\Pic((X_{\bar k})_{sm})$ is trivial. Since $X_{\bar k}$ is normal, one concludes that
$\codim(X_{\bar k}\setminus (X_{\bar k})_{sm},  X_{\bar k})$ $\ge 2$, then we have $\Pic(X_{\bar k})\hookrightarrow \Cl(X_{\bar k})=\Pic((X_{\bar k})_{sm})$. Hence $\Gamma_k$ acts trivially on $\Pic(X_{\bar k})$.

 By \cite[Corollary 2.3.9]{Sko}, we have the following exact sequence
$$
  \begin{CD}
    0 @>>>\Pic(X) @>>> \Pic(X_{\bar k}) @>>>\Br(k)@>{p}>>\Br_1(X) \\
    @. @VVV @V V  V @VVV @VVV \\
    0 @>>>\Pic(X_{sm}) @>>> \Pic((X_{\bar k})_{sm})@>>>\Br(k) @>{p'}>> \Br_1(X_{sm}).
  \end{CD}
$$

Since $X_{sm}(k)\ne \emptyset$, the maps $p$ and $p'$ are injective, one concludes that
$$\Pic(X_{sm})\cong \Pic((X_{\bar k})_{sm}), \ \  \Pic(X)\cong \Pic(X_{\bar k}).$$
Therefore, there is a line bundle $\mathcal{F}$ on $X_{sm}$ which is isomorphic to $\overline{\mathcal{F}}$ over $\bar{k}.$ Note that $\mathrm{O}(m)$ is an ample line bundle on $\mathbb{P}_{\bar{k}}(q_0, q_1,...,q_n)$ by \cite[Lemma 1.3]{Mori}, where $m=\lcm(q_0, q_1,\cdots, q_n).$ Then there is a line bundle $\mathcal{L}$ on $X$ which is isomorphic to $\mathrm{O}(m)$ over $\bar{k}$ and we have $$\mathcal{L}|_{X_{sm}}\cong \mathcal{F}^{\otimes{m}}.$$
Since $X$ is normal, we have $\codim(X\setminus X_{sm},  X)\ge 2,$ one concludes that
\begin{equation}\label{A.2}
H^0(X,\mathcal{L}^{\otimes{l}})=H^0(X_{sm},\mathcal{F}^{\otimes{lm}})
\end{equation}
for ${l\in \Bbb Z}$. Note that $H^0(\mathbb{P}_{\bar{k}}(q_0, q_1,...,q_n)_{sm},\overline{\mathcal{F}}^{\otimes{l}})=H^0(X_{sm},\mathcal{F}^{\otimes{l}})\otimes_k\bar{k},$ then we have
 $$\bigoplus_{l\in \Bbb Z}H^0(X_{sm},\mathcal{F}^{\otimes{l}})\otimes_k{\bar{k}}\cong \bar{k}[X_0,\cdots, X_n]$$
by (\ref{A.1}). This also implies that $$\bigoplus_{l\in \Bbb Z}H^0(X_{sm},\mathcal{F}^{\otimes{l}})\cong k[X_0,\cdots, X_n].$$
 Hence we have
$$\bigoplus_{l\in \Bbb Z}H^0(X,\mathcal{L}^{\otimes{l}})\cong k[X_0,\cdots, X_n]^{(m)}$$
by (\ref{A.2}). This induces a $k$-rational map $X \dashrightarrow \Proj(k[X_0,\cdots, X_n]^{(m)})\cong \mathbb{P}_k(q_0, q_1,...,q_n)$ (cf. \cite[\S 3.7]{EGAII}),
which becomes an isomorphism
over $\bar{k},$ hence $X\cong \mathbb{P}_{k}(q_0, q_1,...,q_n)$.
\end{proof}
\begin{remark}
When $q_0=\cdots=q_n=1$, Proposition \ref{prop.2.2.0} is a theorem of Ch\^atelet (see \cite[Theorem 5.1.3]{GS06}).
\end{remark}
\begin{corollary}\label{k-form}
Let $X$ be a $k$-form of $\mathbb{P}_k(q_0, q_1,...,q_n)$ with $q_0=1$. If $X_{sm}(k)\ne \emptyset$, then $X_{sm}$ satisfies the arithmetic purity.
\end{corollary}
\begin{proof}
It follows from Proposition \ref{prop.2.1} and \ref{prop.2.2.0}.
\end{proof}

\section{Proofs of the main results}

\subsection{Proof of Theorem \ref{t.2.3}}
As in the proof of \cite[Theorem 1.1]{Wei14}, we use the descent theory and combinatorial description of toric varieties.
Let $\widehat{T}$ be the character group of $T$ which is a $\Gamma_k$-module. Denote
\begin{equation}\label{a}
 N := \Hom(\widehat{T},\Bbb Z)
\end{equation}
which has a $\Gamma_k$-action such that the pairing$\langle \cdot , \cdot \rangle: N \times \widehat{T} \to \Bbb Z$ is $\Gamma_k$-invariant. Set $N_{\Bbb R}:= N \otimes \Bbb R$ and let $\Sigma \subset N_{\Bbb R}$ be the fan of $X$ which is $\Gamma_k$-invariant, i.e., for any $\sigma \in \Sigma$ and $g \in \Gamma_k$, we have $g(\sigma) \in \Sigma$.

Since $X$ is normal, we have $\codim(X\setminus X_{sm}, X)\ge 2$ and hence $\bar{k}[X_{sm}]^\times={\bar{k}}^\times$.
Let $\Div_{T_{\bar k}}((X_{\bar k})_{sm})$ ($\cong \Div_{T_{\bar k}}(X_{\bar k})$) be the group of $T_{\bar k}$-invariant Weil divisors of $(X_{\bar k})_{sm}$. By \cite[Theorem 4.1.3]{CJH}, we have the following exact sequence
$$0\to \widehat{T} \xrightarrow{\text{div}} \Div_{T_{\bar k}}((X_{\bar k})_{sm}) \to \Pic((X_{\bar k})_{sm}) \to 0. $$
By \cite[Proposition 6.1.4]{Sko}, the universal torsors of $X_{sm}$ exist. Let $f:\mathscr{F}\to X_{sm}$ be a universal torsor of $X_{sm}$. Let $\widehat{M} :=\Div_{T_{\bar k}}((X_{\bar k})_{sm})$ and $M$ the dual torus of $\widehat{M}$. Let $i: M\to T$ be the morphism induced by $\text{div}: \widehat{T} \to \widehat{M}$. By \cite[Theorem 4.3.1]{Sko},
the restriction of the universal torsor $\mathscr{F}$ to $T$ has the form $\mathscr{F}_T  = M \times_T T$, which is the pull-back of $i: M\to T$ and $\tau: T\to T$, where $\tau$ is a translation of $T$ induced by a splitting of the following exact sequence
$$0\to \bar{k}^\times \to \bar{k}[T]^\times \to \bar{k}[T]^\times/\bar{k}^\times \to 0.$$
We have the following commutative diagram
$$
  \begin{CD}
    M \times_T T@>{p_1}>> M\\
    @V V {p_2}V @V V i V \\
    T @>{\tau}>> T,
  \end{CD}
$$
where $p_1$ and $p_2$ are the natural projections. Since $\tau$ is an isomorphism, we may replace $\mathscr{F}_T$ by $M$ with the structure morphism $\tau^{-1}\circ i: M\to T$.

Since $Z$ has codimension $\ge 2$, we have
$$\bar{k}[X_{sm}\setminus Z]^\times=\bar{k}[X_{sm}]^\times={\bar{k}}^\times,\ \Pic((X_{\bar k})_{sm}\setminus Z_{\bar{k}})=\Pic((X_{\bar k})_{sm}).$$

Let $S$ be a group of multiplicative type. By \cite[Corollary 2.3.9]{Sko}, we have the following commutative diagram
\begin{equation}
  \begin{CD}
    0 @>>> H^1(k,S)@>>> H^1(X_{sm},S) @>{\text{type}}>> \Hom(\Shat, \Pic((X_{\bar k})_{sm})@>>> H^2(k,S)\\
    @. @| @V V \Res V @VV{\cong}V @|@.\\
    0 @>>> H^1(k,S)@>>> H^1(X_{sm}\setminus Z,S) @>{\text{type}}>> \Hom(\Shat, \Pic((X_{\bar k})_{sm}\setminus  Z_{\bar{k}} ))@>>> H^2(k,S).
  \end{CD}
\end{equation}

Then any universal torsor $\mathscr{F}'$ of $X_{sm}\setminus Z$ is the restriction
of a universal torsor $f:\mathscr{F}\to X_{sm}$ to $X_{sm}\setminus Z$. Hence the restriction of $\mathscr{F}'$ to $T\setminus Z$ is $M\setminus (i^{-1}\circ \tau)(Z)$.

We shall call a 1-dimensional cone a ray. Let $\Sigma (1)$ be the set of rays of $\Sigma$ and set $d=\#\Sigma (1)$. We write
 $$ \Sigma (1)=\{\rho_i| 1\leqslant i\leqslant d\}.$$
Let $u_i \in N$ be the minimal generator of $\rho_i$ for $1\leqslant i\leqslant d$ (cf. \cite[p. 29, 30]{CJH}). Set
\begin{equation}\label{A.}
 e=-A\sum_{i=1}^{d}{u_i},
\end{equation}
where $A\in {\Bbb Z}_{>0}$. Obviously, $e$ is $\Gamma_k$-invariant. Since $X$ is complete, we have $N_{\Bbb R}=\bigcup_{\sigma \in \Sigma}\sigma$ by \cite[Theorem 3.4.1]{CJH}. Without loss of generality, we assume that $\sigma(u_j,u_{j+1},\cdots,u_{d-1},u_d)$ is a $\Gamma_k$-invariant cone containing $e$ (such cone exists, for example, one can take the smallest cone containing $e$), where $j>1$ and $u_j,u_{j+1},\cdots,u_{d-1},u_d$ are its  minimal generators (cf. \cite[p. 29, 30]{CJH}). Then we have
\begin{equation}\label{A..}
e=\sum_{t=j}^{d}c_t{u}_t,
\end{equation}
where $c_t \in {\Bbb R}_{\ge0}$ for $j \leqslant t\leqslant d$. Without loss of generality, we can choose $c_t \in {\Bbb Z}_{\ge0}$ for $j \leqslant t\leqslant d$. In fact, the equation $e=\sum_{t=j}^{d}X_t{u}_t$ defines a linear subvariety of $\Spec(\Bbb{Q}[X_j,\cdots, X_d])$, which is isomorphic to an affine space. By using weak approximation of affine space, one can take $c_t \in {\Bbb Q}_{\ge0}$ for $j \leqslant t\leqslant d$. Then by taking a suitable $A\in {\Bbb Z}_{>0}$, we can assume $c_t \in {\Bbb Z}_{\ge0}$ for $j \leqslant t\leqslant d$. In particular, if $T$ is a $k$-anisotropic torus, we have $e=0$, since $N^{\Gamma_k}=0$ (see (\ref{a})). In this case, we take $c_t=0$ for $j \leqslant t\leqslant d.$

 By (\ref{A.}) and (\ref{A..}), we have
\begin{equation}\label{A...}
A\sum_{i=1}^{d}{u}_i+\sum_{t=j}^{d}c_t{u}_t=0.
\end{equation}

Let $D_i$ be the $T_{\bar k}$-invariant weil divisor associated to $\rho_i \in \Sigma (1)$ (cf. \cite[Section 4.1]{CJH}).
Recall $\widehat{M}=\bigoplus_{i=1}^{d}{\Bbb Z}D_i=\bigoplus_{i=1}^{d}{\Bbb Z}u_i$. Define $\widetilde{N}:=X_*(M)=\Hom_{\Bbb Z}(\widehat{M}, \Bbb Z)$ which has a $\Gamma_k$-action such that the pairing
$$\langle \cdot , \cdot \rangle: \widetilde{N} \times \widehat{M} \to \Bbb Z$$
is $\Gamma_k$-invariant. Let $\widetilde{D}_i \in \widetilde{N}$ be the dual of $D_i$ for $ 1\leqslant i\leqslant d$ and set $\widetilde{N}_{\Bbb R}:=\widetilde{N}\otimes {\Bbb R}$. We write
\begin{equation}\label{D}
  \widetilde{D}_0=-(A\sum_{i=1}^{d}\widetilde{D}_i+\sum_{t=j}^{d}c_t\widetilde{D}_t) \in \widetilde{N}.
\end{equation}
Now we claim $\widetilde{D}_0$ is $\Gamma_k$-invariant.

(1) If $X$ is simplicial, the generators $u_j,u_{j+1},\cdots,u_{d-1},u_d$ are linearly independent over $\Bbb R$ by \cite[Theorem 3.1.19, Definition 1.2.16 and 3.1.18]{CJH}. Thus $\widetilde{D}_0$ is $\Gamma_k$-invariant since $e$ is $\Gamma_k$-invariant (see (\ref{A..})).

(2) In the case $T$ is a $k$-split torus, the action of $\Gamma_k$ on $\widetilde{N}$ is trivial. Thus $\widetilde{D}_0$ is $\Gamma_k$-invariant (see (\ref{D})).

(3) In the case $T$ is a $k$-anisotropic torus, we have taken $c_t=0$ for $j \leqslant t\leqslant d.$ Thus $\widetilde{D}_0$ is $\Gamma_k$-invariant.

Now let $\Delta$ be the fan in $\widetilde{N}_{\Bbb R}$ made up of the cones generated by all the proper subsets of $\{\widetilde{D}_0, \widetilde{D}_1, \cdots, \widetilde{D}_d\}$ (cf. \cite[Example 3.1.17]{CJH}). The fan $\Delta$ is $\Gamma_k$-invariant by our constructions. Denote by $Y$ the $M$-toric variety associated to $\Delta$. Note that $Y$ is defined over $k$ by the effectiveness of the descent (see \cite[Example B in pp.139-141]{Neron} or \cite[Chapter V, \S 4, Corollary 2]{Serre}). By Corollary \ref{k-form}, the variety $Y_{sm}$ satisfies the arithmetic purity.

By \cite[Lemma 3.1]{DW}, to prove this theorem, it suffices to prove the following claim:

There is a closed subscheme $W\subset Y_{sm}$ of codimension $\ge 2$, satisfying $Y_{sm}\setminus W \supset M\setminus(i^{-1}\circ \tau)(Z)$, and such that the restriction of $\tau^{-1}\circ i$ to $M\setminus (i^{-1}\circ \tau)(Z)$ can be extended to $Y_{sm}\setminus W \to X_{sm}\setminus Z$.

Note that the translation $\tau: T \to T$ can be extended to a unique isomorphism of $X_{sm}$ which we also denote by $\tau$.
In fact, we only need to show that there exists a closed subscheme $W\subset Y_{sm}$ of codimension $\ge 2$ such that $i: M\setminus i^{-1}(\tau (Z)) \to T\setminus \tau(Z)$ can be extended to $Y_{sm}\setminus W \to X_{sm}\setminus \tau(Z)$. Then the composite map $Y_{sm}\setminus W \to X_{sm}\setminus \tau(Z)\xrightarrow{{\tau}^{-1}} X_{sm}\setminus Z$ gives the desired map.

Now we show that such $W$ exists. Define the $\Gamma_k$-invariant subfan of $\Delta$
$$\Delta':=\{0\}\cup \Delta (1)\subset \Delta,$$
where $\Delta (1)$ is the set of rays of $\Delta$. Denote by $\Delta'(1)$ the set of rays of $\Delta'$ and by $Y_{\Delta'}$ the $M$-toric variety associated to $\Delta'$. Similarly, set
$$\Sigma':=\{0\}\cup \Sigma (1)\subset \Sigma.$$
Denote by $X_{\Sigma'}$ the toric variety associated to $\Sigma'$. Note that the toric varieties $X_{\Sigma'}$ and $Y_{\Delta'}$ are defined over $k$ by Galois descent. By \cite[Theorem 3.1.19]{CJH}, we have $X_{\Sigma'}\subset X_{sm}$ and $Y_{\Delta'}\subset Y_{sm}$. By \cite[Theorem 3.2.6 and Exercise 4.1.1]{CJH}, we have
\begin{center}
$\codim(X_{sm}\setminus X_{\Sigma'}, X_{sm})\ge 2$ and $\codim(Y_{sm}\setminus Y_{\Delta'}, Y_{sm})\ge 2.$
\end{center}

We will prove the claim as follows. We construct a morphism $g: Y_{\Delta'}\to X_{sm}$ which extends $i: M\to T$, such that codim$(g^{-1}
(\tau(Z)), Y_{\Delta'})\ge 2$. Then we can take $W=\overline{(Y_{sm}\setminus Y_{\Delta'})\bigcup g^{-1}(\tau(Z))}\subset Y_{sm}.$

Let $\tilde{g}:\widetilde{N}_{\Bbb R}\to {N}_{\Bbb R}$ be the morphism induced by $\text{div}: \widehat{T} \to \widehat{M}$. By \cite[Proposition 4.1.2]{CJH}, we have
$$\text{div}: \widehat{T} \to \widehat{M}, \ \ \chi \mapsto \text{div}(\chi)=\sum_{i=1}^{d}\langle u_i, \chi \rangle D_i.$$
For $1\leqslant i\leqslant d$, we have $\tilde{g}(\widetilde{D}_i)\in N=\Hom(\widehat{T},\Bbb Z)$ and for any $\chi \in \widehat{T}$
$$\tilde{g}(\widetilde{D}_i)(\chi)=\widetilde{D}_i(\sum_{i=1}^{d}\langle u_i, \chi \rangle D_i)=\langle u_i, \chi  \rangle .$$
Then $\tilde{g}(\widetilde{D}_i)=u_i$ for $1\leqslant i\leqslant d$ and $\tilde{g}(\widetilde{D}_0)=-(A\sum_{i=1}^{d}{u}_i+\sum_{t=j}^{d}c_t{u}_t)=0$ by (\ref{A...}). Hence $\tilde{g}$ induces a morphism of fans
$\Delta' \to \Sigma'$ compatible with the action of $\Gamma_k$. This gives a toric morphism $g:Y_{\Delta'}\to X_{\Sigma'} \to X_{sm}$ over $k$ which extends $i:M \to T$.

Now we prove that codim$(g^{-1}(\tau(Z)), Y_{\Delta'})\ge 2$. Note that $i:M \to T$ is surjective, hence it's flat. Then we have codim$(i^{-1} (\tau(Z)\cap T), M)\ge 2$. On the other hand, by \cite[Theorem 3.2.6 and Lemma 3.3.21]{CJH}, one can see that the divisor of $Y_{\Delta'}$ associated to Cone$(\widetilde{D}_i) \in  \Delta' (1)$ dominates the divisor $D_i$ of $X_{sm}$ for $1\leqslant i\leqslant d$ and the divisor of $Y_{\Delta'}$ associated to Cone($\widetilde{D}_0) \in \Delta' (1)$ dominates $X_{sm}.$ Thus we have codim$(g^{-1}(\tau(Z)), Y_{\Delta'})\ge 2$.
\qed

\begin{corollary}\label{cor.2.4.1}
For a general weighted projective space $\mathbb{P}_k(q_0, q_1,...,q_n)$, its smooth locus satisfies the arithmetic purity.
\end{corollary}
\begin{proof}
By Lemma \ref{lem.2.0.1} and Remark \ref{rem 2.0.1.1}, one has $\Pic(\mathbb{P}_{\bar k}(q_0, q_1,...,q_n)_{sm})\simeq \Bbb Z.$ By \cite[Corollary 2.3.9]{Sko}, we have
 $$\Br_a(\mathbb{P}_k(q_0, q_1,...,q_n)_{sm})\cong H^1(k, \Bbb Z)=0.$$
Now it follows from Theorem \ref{t.2.3}.
\end{proof}
\begin{remark}
By Proposition \ref{prop.2.2.0}, Corollary \ref{cor.2.4.1} is also true for $k$-forms of weighted projective spaces with smooth $k$-points.
\end{remark}
\begin{corollary}
Let $V \subset \mathbb{P}_k^3$ be the cubic surface over $k$ defined by equation
$$x_0x_1x_2=t^{3}.$$
Then $V_{sm}$ satisfies the arithmetic purity with respect to $\Br_1(X_{sm})$.
\end{corollary}
\begin{proof}
It's well known that $X$ is a toric singular del Pezzo surface with singularity of type $3{\mathbf A}_2$. One can apply Theorem \ref{t.2.3}
to this surface.
\end{proof}
\begin{remark}

(1) Note that the torus $\Res^{(1)}_{K/k}\Gm$ is $k$-anisotropic, where $K=k(\omega)$ is a cubic extension over $k$. Then one can also apply Theorem \ref{t.2.3} to the cubic surface $V' \subset \mathbb{P}_k^3$ defined by equation
$$N_{K/k}(x_0+x_1\omega+x_2{\omega}^2)=t^{3}.$$

(2) More examples of toric singular del Pezzo surfaces can be found in \cite{D}.
\end{remark}

\begin{corollary}
Let $G$ be a reductive group over $k$, and $P$ be a parabolic subgroup of $G$ over $k$. Let $X'$ be a complete, smooth $T$-toric variety over $k$. Let $\lambda : P \to T$ be a homomorphism which defines an action of $P$ on $X'$. Set $Q:= G\times^{P} X'$. Then $Q$ satisfies the arithmetic purity with respect to $\Br_1(Q)$.
\end{corollary}
\begin{proof}
Let $Z\subset Q$ be a closed subset of codimension $\ge 2$. We need to prove that $Q\setminus Z$ satisfies strong approximation with respect to $\Br_1(Q)$. We will denote by $\overline{(g, x)}$ the image of the pair $(g, x)$ under the quotient map $G\times X' \to G\times^{P} X'$. Let
$$p_0: Q \to G/P ; \ \ \overline{(g, x)} \mapsto gP.$$
Note that the map $p_0$ has local sections (cf. \cite[\S5.5.7, \S 5.5.8 and Theorem 15.1.3]{Spr}), and its fibers lying over rational points of $G/P$ are isomorphic to $X'$.
Consider the map
$$p_1: Q\setminus Z \to  G/P$$
induced by $p_0$. We claim that there exists an open subset $U\subset G/P$ such that for any $\theta \in U(k)$, the fiber $p_{1}^{-1}(\theta)$ is an open subset of $X'$ with complement of codimension $\ge 2$. If $\overline{p_0(Z)}$ is a proper subset of $G/P$, then we can take $U=(G/P)\setminus \overline{p_0(Z)}$. If $\overline{p_0(Z)}=G/P$, we consider the map $p_2: Z \to G/P$. Then there is an open subset $V\subset G/P$ such that for any $\theta \in V(k)$, we have $\codim (p_2^{-1}(\theta), p_0^{-1}(\theta))\ge 2$ by \cite[Chapter II, Exercise 3.22(e)]{Hartshorne}. Then we can take $U=V$. Now one can check that the map $p_1$ satisfies all conditions of \cite[Theorem 4.3.1]{Hara}. By the proof of \cite[Theorem 4.3.1]{Hara}, we obtain that $Q\setminus Z$ satisfies strong approximation with respect to $\Br_1(Q)$.

\end{proof}

\subsection{Proof of Theorem \ref{t.2.5}}
By \cite[Lemma 3.2]{PR} and valuative criterion of properness, one concludes that $X(k)$ is dense in $X({\mathbf A}^S_k)$. Now it follows from Theorem \ref{t.2.3} and Lemma \ref{lem.2.6}.
\qed

\begin{lemma}\label{lem.2.6}
Let $k$ be a number field and $S$ be a finite set of places of $k$. Let $X$ be a smooth and geometrically integral variety such that $\Br_a(X)$ is finite. Let $U$ be an open set of $X$. Suppose that\\

(1) $X(k)$ is dense in $X({\mathbf A}_k)^{\Br_a(X)}$,

(2) $X(k)$ is dense in $X({\mathbf A}^S_k)$,

(3) $U(k)$ is dense in $U({\mathbf A}_k)^{\Br_a(X)}$.\\

Then $U(k)$ is dense in $U({\mathbf A}^S_k)$.
\end{lemma}
\begin{proof}
We use the same notation $pr^S$ for the projections $X({\mathbf A}_k) \to X({\mathbf A}^S_k)$ and $U({\mathbf A}_k) \to U({\mathbf A}^S_k)$. Since $\Br_a(X)$ is finite, we obtain that $pr^S(X({\mathbf A}_k)^{\Br_a(X)})$ is closed in $X({\mathbf A}^S_k)$ and $pr^S(U({\mathbf A}_k)^{\Br_a(X)})$ is closed in $U({\mathbf A}^S_k)$. By (3), we obtain that $U(k)$ is dense in $pr^S(U({\mathbf A}_k)^{\Br_a(X)})$, hence we only need to prove $pr^S(U({\mathbf A}_k)^{\Br_a(X)})=U({\mathbf A}^S_k)$. By (1) and (2), one has $pr^S(X({\mathbf A}_k)^{\Br_a(X)})=X({\mathbf A}^S_k)$. Therefore, for any $(P_v)_{v\in {\Omega_k\setminus S}} \in U({\mathbf A}^S_k) \subset X({\mathbf A}^S_k)$, there exists $(P_v)_{v \in S} \in \prod_{v\in S}X(k_v)$ such that
$$ \sum_{v \in S}\zeta(P_v) + \sum_{v \in {\Omega_k \setminus S}}\zeta(P_v)=0$$
for any $\zeta \in \Br_a(X)$. Since $\Br_a(X)$ is finite, one can take $(P_v)_{v \in S} \in \prod_{v \in S}U(k_v)$ by \cite[Lemma 3.2]{PR} and \cite[Proposition 8.2.9]{Poonen}. Thus $pr^S(U({\mathbf A}_k)^{\Br_a(X)})=U({\mathbf A}^S_k)$.
\end{proof}

\bf{Acknowledgments} 	
 \it{ The author would like to thank Yang Cao for many fruitful discussions, and to Fei Xu and Yongqi Liang for useful suggestions. The author is partially supported by National Natural Science Foundation of China No.12071448 and Anhui Initiative in Quantum Information Technologies No. AHY150200.}


\end{document}